\documentclass[reqno,11pt]{amsart}
\usepackage{amsmath}
\usepackage{amsthm}
\usepackage{amssymb}
\usepackage{xypic}
\usepackage{mathrsfs}
\usepackage{hyperref}
\usepackage{graphicx}
\usepackage{perpage}
\usepackage{extarrows}
\MakePerPage{footnote}
\date{\today}
          
 \usepackage[usenames,dvipsnames]{pstricks}
\usepackage[shell]{pdftricks}
 \begin{psinputs}
 \usepackage{pst-grad} 
 \usepackage{pst-plot} 
 \usepackage{epsfig}
 \end{psinputs}

\title[Bi-Lipschitz geometry of contact orbits]{Bi-Lipschitz geometry of contact orbits in the boundary of the nice dimensions}

\author{Maria Aparecida Soares Ruas \& Saurabh Trivedi}

\begin{document}

\newtheorem{thm}{Theorem}[section]
\newtheorem{lem}[thm]{Lemma}
\theoremstyle{theorem}
\newtheorem{prop}[thm]{Proposition}
\newtheorem{defn}[thm]{Definition}
\newtheorem{cor}[thm]{Corollary}
\newtheorem{example}[thm]{Example}
\newtheorem{xca}[thm]{Exercise}

\newcommand{\bb}{\mathbb}
\newcommand{\al}{\mathcal}
\newcommand{\ak}{\mathfrak}
\newcommand{\fs}{\mathscr}

\theoremstyle{remark}
\newtheorem{rem}[thm]{Remark}
\parskip .12cm

\begin{abstract} 
Mather proved that the smooth stability of smooth maps between manifolds is a generic condition if and only if the pair of dimensions of the manifolds are 'nice dimensions' while topologically stability is a generic condition in any pair of dimensions. And, by a result of du Plessis and Wall $C^1$-stability is also a generic condition precisely in the nice dimensions. We address the question of bi-Lipschitz stability in this article. We prove that the Thom-Mather stratification is bi-Lipschitz contact invariant in the boundary of the nice dimensions. This is done in two steps: first we explicitly write the contact unimodular strata in every pair of dimensions lying in the boundary of the nice dimensions and second we construct bi-Lipschitz vector fields whose flow provide the bi-Lipschitz contact trivialization in each of the cases.
\end{abstract}

\maketitle

\section{Introduction}

We know from Mather \cite{MatherV} that given a pair of positive integers $(n,p)$, there exists a smallest Zariski closed $\fs K^k$-invariant set $\pi^k(n,p)$ in the set $J^k(n,p)$ of $k$-jets of smooth mappings from $(\bb R^n,0)$ to $(\bb R^p,0)$ such that the complement of $\pi^k(n,p)$ in $J^k(n,p)$ is the union of only finitely many $\fs K^k$-orbits. 

The set $\pi^k(n,p)$, let's say the bad set, is in fact the set of $k$-jets in $J^k(n,p)$ of `modality' greater than or equal to $1$. The codimension of $\pi^k(n,p)$ in $J^k(n,p)$ decreases as $k$ increases. We know that there exists a $k$ big enough for which the codimension of the bad set attains its minimum. For this $k$ the codimension of the bad set $\pi(n,p)$ is denoted by $\sigma(n,p)$. Mather \cite{MatherVI} calculated the number $\sigma(n,p)$ and the results he obtained are as follows:
\begin{align*}
\intertext{\bf Case 1: $n \leq p$}
\sigma(n,p) &= \begin{cases}
6(p-n) + 8, & \text{if } p - n \geq 4 \text{ and } n \geq 4\\
6(p-n) + 9, & \text{if } 3 \geq p-n \geq 0 \text{ and } n \geq 4 \text{ or if } n =3\\
7(p-n)+10, & \text{if } n =2\\
\infty, & \text{if } n=1  
\end{cases}
\intertext{\bf Case 2: $n > p$}
\sigma(n,p) &= 
\begin{cases}
9, & \text{if } n = p+1\\
8, & \text{if } n = p+2\\
n-p+7,& \text{if } n \geq p+3
\end{cases}
\end{align*}

Suppose $k$ has the property that the bad set has codimension $\sigma(n,p)$. The complement of the bad set in $J^k(n,p)$ is also a $\fs K^k$-invariant set and the orbits lying in this set are called the $\fs K$-simple orbits. More is known about the complement of the bad set $\pi(n,p)$ in $J^k(n,p)$. It is not only the union of finitely many $\fs K^k$-orbits but also these $\fs K^k$-orbits are disjoint submanifolds of $J^k(n,p)$. Thus, these finitely many $\fs K^k$-orbits stratify the complement of the bad set. In fact, the stratification thus obtained is smoothly trivial. 

Mather \cite{MatherV} proved that the number $\sigma(n,p)$ is greater than $n$ if and only if the set of proper smoothly stable mappings between any smooth manifolds $N$ and $P$ of dimensions $n$ and $p$ respectively is dense in the space $C^\infty_{pr}(N,P)$ of all proper smooth mappings between $N$ and $P$ with Whitney strong topology. The pairs of integers $(n,p)$ for which $\sigma(n,p) > n$ are called the nice dimensions. And, we say that the pairs of integers $(n,p)$ are in the boundary of nice dimensions if $\sigma(n,p) = n$. Later, Mather \cite{Mather, Mather7} proved that the proper topologically stable maps between $N$ and $P$ are dense in the set of proper smooth maps with the Whitney strong topology without any restrictions on the dimensions on $N$ and $P$. This result is called the topological stability theorem and is perhaps one of the deepest results in singularity theory. A detailed exposition and much more about the stratifications of jet spaces can be found in the book of du Plessis and Wall, \cite{duplessis3}.

Much later du Plessis and Wall \cite{duplessis2} proved that outside the nice dimensions ($\sigma(n,p) \leq n$) the set of proper smooth mappings between any manifolds $N$ and $P$ of dimensions $n$ and $p$ respectively contains a strongly open set of non $C^1$-stable maps. Since proper smoothly stable maps in the nice dimensions are dense, this implies that $\sigma(n,p) > n$ is a necessary and sufficient condition for proper $C^1$-stable maps to be dense. 

It is natural to ask whether the set of proper Lipschitz stable maps are dense in the set of all maps between two smooth manifolds $N$ and $P$ with the Whitney strong topology outside the nice dimensions. We conjecture that in the boundary of the nice dimensions such a result is true. In this article we propose a line of attack to this conjecture. Namely, we describe explicitly the Thom-Mather stratification in the boundary of the nice dimensions. The strata of the stratification are the $\fs K^k$-orbits of smooth simple jets with a distinguished stratum made of jets of codimension equal to the dimension of the source that have smooth $\fs K$-modality $1$. We describe this distinguished stratum in every pair $(n,p)$ in the boundary of the nice dimension and show that it is bi-Lipschitz $\fs K^k$-invariant. Since the simple $\fs K^k$-orbits are already bi-Lipschitz invariant, this shows that every stratum and so the Thom-Mather stratification is bi-Lipschitz $\fs K^k$-invariant.  

\section{Thom-Mather stratification in the boundary of the nice dimensions}

In the proof of the topological stability theorem Mather \cite{Mather, Mather7} showed that for any pair of integers $(n,p)$ and for a large $k$ there exists a set $S \subset J^k(n,p)$ with the following properties:

1. $S$ has codimension greater than $n$.

2. There is a stratification of $J^k(n,p) - S$. 

\noindent
The stratification of $J^k(n,p) - S$ is called the Thom-Mather stratification. It is then shown that any proper map whose $k$-jet extension is transverse to the Whitney stratification of $J^k(n,p) - S$ and avoids $S$ is topologically stable. The topological stability theorem follows from the fact that such proper transverse maps are dense in the set of all proper maps. This last fact is proved using Thom's second isotopy lemma.

In the nice dimensions this set $S$ can be taken to be the bad set $\pi(n,p)$ and its complement can be stratified by $\fs K^k$-orbits. Outside the nice dimensions it is difficult to say much about the set $S$ due to the nature of its construction but in the boundary of the nice dimensions we do have an explicit description of the set $S$ and the strata of the complement of $S$ in $J^k(n,p)$. To describe $S$ we first recall the $\fs K$-codimension of a germ $f : (\bb R^n,0) \to (\bb R^p,0)$. The notations are by now standard in Singularity theory. Denote by $\al E_n$ the local ring of germs of smooth functions on $\bb R^n$ and $\ak m_n$ its maximal ideal. Denote by $\theta(f)$ the set of germs of vector fields along $f$, it can be identified with a module of rank $p$ over $\al E_n$. Denote by $\theta(n)$ the set of vector fields along the identity map germ $id : (\bb R^n,0) \to (\bb R^n,0)$. 

Then the $\fs K$-tangent space of $f$ is defined to be the module $T{\fs K}(f) = tf(\ak m_n\theta(n)) + f^*(\ak m_p)\theta(f)$, where $tf(\ak m_n \theta(f))$ is the Jacobian module of $f$ and $f^*(\ak m_p)$ is the ideal of $\al E_n$ generated by the components of $f$. Then, the $\fs K$-codimension of $f$ is defined to be the $\bb R$-vector space dimension of $\theta(f)/T\fs K(f)$. Given a germ $f : (\bb R^n,0) \to (\bb R^p,0)$, the $\fs K$-codimension is equal to the $\fs K^k$-codimension of the $k$-jet of $f$ under the canonical action of contact group on $J^k(n,p)$ for $k$ large enough. 

We also need the notion of modularity\footnote{Modularity can be defined for any group $\fs R, \fs L, \fs A, \fs C, \fs K$, in this article by modularity we mean for group $\fs K$.}. Let $z \in J^k(n,p)$, denote by $\fs K^*(z)$ the union of all $\fs K^k$-orbits of codimension equal to the codimension of $z$ in $J^k(n,p)$. The set $\fs K^*(z)$ in $J^k(n,p)$ is a Zariski closed subset. Suppose $\fs K_*(z)$ is the connected component of $\fs K^*(z)$ in which $z$ lies. Then, we say that $z \in J^k(n,p)$ is \emph{$r$-modular} if:
$$cod \fs K_*(z) = cod \fs K^k.z - r.$$
We can say that $1$-modular jets are unimodular, $2$-modular jets are bimodular and so on. Also, if the union of unimodular jets is a submanifold of $J^k(n,p)$, which it will be in our case, we call this union a unimodular stratum. 

We know that the bad set $\pi(n,p)$ consists of the jets of modularity greater than or equal to 1. We take the unimodular strata from the bad set and add them to its complement in $J^k(n,p)$. Since in the boundary of the nice dimensions these are only strata of codimension $n$, the rest of the bad set, let's say $\tilde{\pi}(n,k)$, will have codimension greater than $n$ and in this case $S$ can be taken to be precisely this set $\tilde{\pi}(n,k)$. We stratify $J^k(n,p) - \tilde{\pi}(n,k)$ by taking strata as the $\fs K^k$-orbits of the stable maps and the unimodular strata. Let us call this stratification $\Sigma_{bnd}$ ($bnd$ for boundary of the nice dimensions).

In the global setting we have the following situation. Let $N$ and $P$ be two smooth manifolds of dimensions $n$ and $p$ respectively and denote by $J^k(N,P)$ the $k$-jet bundle over $N \times P$ (with fibers $J^k(n,p)$). Denote by $\tilde{\pi}(N,P)$ the subbundle of $J^k(N,P)$ with fibers $\tilde{\pi}(n,p)$. Then, clearly the codimension of $J^k(N,P) \setminus \tilde{\pi}(N,P)$ in $J^k(N,P)$ is equal to the codimension of $\tilde{\pi}(n,p)$ in $J^k(n,p)$. Moreover the stratification $\Sigma_{bnd}$ induces a stratification on $J^k(N,P) \setminus \tilde{\pi}(N,P)$, we call this $\Sigma_{bnd}(N,P)$. 

\section{Whitney $(a)$-regularity of $\Sigma_{bnd}$}

It is known that the Thom-Mather stratification is Whitney $(b)$-regular. The proof of this is very complicated and occupies bulk of the last chapter in Gibson et al. \cite{GWPL}. We give an easy proof that the constructed stratification $\Sigma_{bnd}$ of $J^k(n,p) - \tilde{\pi}(n,p)$ in the boundary of the nice dimensions is Whitney $(a)$-regular\footnote{Although the result is weaker than the known result (Whitney $(b)$-regularity), the idea of the proof can have other applications.}. For this we invoke a theorem of Trotman \cite{Trotman} which characterizes Whitney $(a)$-regular stratifications as those for which the maps transverse to the stratification forms an open set in the Whitney strong topology. 

Let $N$ and $P$ be smooth manifolds of dimension $n$ and $p$ respectively such that $(n,p)$ lies in the boundary of nice dimensions. Let $N$ be a compact manifold. Denote by $J^k(N,P)$, for $k$ large enough, the $k$-jet bundle of smooth maps between $N$ and $P$ and by $\tilde{\Pi}(N,P)$ the subbundle of $J^k(N,P)$ with fiber $\tilde{\pi}(n,p)$. Denote by $\Sigma_{bnd}(N,P)$ the stratification of the subset of $J^k(N,P)$ induced by $\Sigma_{bnd} \subset J^k(n,p)$. We will show that:

\begin{thm}\label{Bndopen}
The set of maps $f : N \to P$ such that $j^kf(N) \cap \tilde{\Pi}(n,p) = \emptyset$ and is transverse to the strata of $\Sigma_{bnd}(N,P)$ is open in the strong topology on the set of smooth maps between $N$ and $P$.
\end{thm}

For this we need the following lemmas:

\begin{lem}
Let $X$ and $Y$ be smooth manifolds, $W \subset Y$ a submanifold and $f : X \to Y$. Let $p \in X$ and $f(p) \in W$. If there is a neighbourhood $U$ of $f(p)$ in $Y$ and a submersion $\phi : U \to \bb R^k$ ($k$ is the codimension of $W$) such that $W \cap U = \phi^{-1}(0)$ then $f$ is transverse to $W$ at $p$ if and only if $\phi \circ f$ is a submersion at $p$.
\end{lem}

And, if codimension of $W$ is equal to the dimension of $X$, one has:

\begin{lem} \label{transverse} (A). Let $X$ and $Y$ be smooth manifolds with $W$ a submanifold of $Y$. Assume that the dimension of $X$ is equal to the codimension of $W$. Let $p$ be in $X$ and let $f : X \to Y$ be smooth. Assume that $f(p)$ is in $W$ and $f$ is transverse to $W$ at $p$. Then, there exist a neighbourhood $\al N$ of $f \in C^{\infty}(X,Y)$ and an open neighbourhood $U$ of $p$ in $X$ such that for all $g$ in $\al N$, $g^{-1}(W) \cap U$ consists of one point $q$ and $g$ is transverse to $W$ at $q$.
	
(B). Assume $X$ is compact. Let $f : X \to Y$ is transverse to $W$. Then, there is an open neighbourhood $\al N$ of $f$ in $C^{\infty}(X,Y)$ such that the number of points in $f^{-1}(W)$ is equal to the number of points in $g^{-1}(W)$ for any $g$ in $\al N$.
\end{lem}

\begin{proof}[Proof of Theorem \ref{Bndopen}] 
	
	Let $U_f$ be an open neighbourhood of $f$ in $C^{\infty}(N,P)$ such that $j^kg(N) \cap \tilde{\Pi}^k(N,P) = \emptyset$ (this is only a $C^0$-condition on the set of maps $h : N \to J^k(N,P)$, since $\pi^k$ is closed). \hspace{\fill} (1)
	
	Applying Lemma \ref{transverse} to $C^{\infty}(N,J^k(N,P))$ where $W$ is the  unimodular stratum which has codimension $n$, one can find a nbhd $V_f$ of $f$ such that for any $g \in V_f$, $j^kg$ is transverse to $W$, and $j^kg^{-1}(W)$ has the same number of points as $(j^kf)^{-1}(W)$.
	
	Now we observe that if $A = \{x_1,\ldots,x_s\} \in N$ is the set $j^kf^{-1}(W)$ and $B = f(A)$ then $f^{-1}(B)$ is a disjoint union of compact manifolds of codimension $p$ (in case $n \geq p$) or a finite number of points (in case $n<p$), and $f : N - f^{-1}(B) \to P - B$ is clearly proper and infinitesimally stable, hence stable. But, since the restriction $f \to f|_{N - f^{-1}(B)}$ is not a continuous mapping in $C^{\infty}$-Whitney topology we cannot in general lift an open set in in $C^{\infty}(N - f^{-1}(B),P - B)$ to an open set in $C^{\infty}(N,P)$.
	
	One may use however, the following lemma:
	
	\begin{lem} Let $f$ be transverse to all strata of $\Sigma_{bnd}^k(N,P)$ and let $K$ be a compact subset of $P$, $K \cap B = \emptyset$, then the set $\Sigma_k$ of $g \in C^{\infty}(N,P)$ with $g|_{f^{-1}K}$ transverse to all strata of $\Sigma_{bnd}^k(N,P)$ is a neighbourhood of $f$ in $C^{\infty}(N,P)$.
	\end{lem}
	
	(the proof of Lemma 4.3 - page 147 of Gibson et al. \cite{GWPL} works for these hypothesis)
	
	Now, we can choose a pair $V_i \subset W_i$ of a locally finite covering of $P - B$ by relatively compact open subsets for which $\overline{V_i} \subset W_i$ for all $i \in I$.
	
	The set $\omega$ of $g \in C^{\infty}(N,P)$ satisfying (1) and $g(f^{-1}(\overline{V_i})) \subset W_i$ for all $i \in I$ is open.
	
	By the above lemma, $\cap_{i\in I}$ $\Omega_{\overline{V_i}}$ is a nbhd of $f$, therefore $\Omega \cap (\cap_{i\in I} \Omega_{\overline{V_i}}) \cap U_f$ gives the result.
\end{proof}

\begin{cor}
	The stratification $\Sigma_{bnd}$ is a Whitney $(a)$-regular stratification of $J^k(n,p) - \tilde{\pi}(n,p)$.
\end{cor}

\begin{proof} We know by the above theorem that the set of maps transverse to the strata of $\Sigma_{bnd}(N,P)$, which is the stratified subbundle of $J^k(N,P)$ with stratified fibers $\Sigma_{bnd}$, is open in the Whitney strong topology. Thus, by the main theorem in Trotman \cite{Trotman}, $\Sigma_{bnd}(N,P)$ is Whitney $(a)$-regular. This implies that $\Sigma_{bnd}$ is also Whitney $(a)$-regular. 
\end{proof}

\section{Unimodular stratum \label{secUnimodular}}

In this section we describe in different cases the unimodular strata in the boundary of the nice dimensions. Using the calculations of $\sigma(n,p)$ as given in the introduction we list all the pairs $(n,p)$ such that $\sigma(n,p) = n$.

For $n \leq p$, we have the following pairs in the boundary of the nice dimensions:

\begin{enumerate}
\item The case $\sigma(n,p)= 6(p-n)+9$ for $3 \geq p-n \geq 0 \text{ and } n \geq 4 \text{ or } n =3$, gives $(n,p) \in \{(9,9),(15,16),(21,23),$ $(27,30)\}$.

\item The case $\sigma(n,p)= 6(p-n)+8$ for $ p - n \geq 4 \text{ and } n \geq 4$, gives $(n,p) \in \{(6t+2,7t+1) : t \geq 5\}$.
\end{enumerate}

For $n > p$, we have the following pairs in the boundary of the nice dimensions:
\begin{enumerate}
\item The case $\sigma(n,p) = 9$ for $n = p+1$, gives $(n,p) = (9,8)$.

\item The case $\sigma(n,p) = 8$ for $n = p+2$, gives $(n,p) = (8,6)$.

\item The case $\sigma(n,p) = n-p+7$ for $n \geq p+3$ gives $(n,p) \in \{(10+k,7) : k\geq 0\}$.
\end{enumerate}

The strategy to find unimodular strata is as follows. Recall that if $F : (\bb R^n,0) \to (\bb R^p,0)$ is an $r$-parameter unfolding of a rank $0$ germ $f : (\bb R^{n-r},0) \to (\bb R^{p-r},0)$, then $F$ is $\fs K$-equivalent to a suspension of $f$ and in this case the $\fs K$-codimension of $F$ and $f$ are the same; see Section 4 of Chapter 4 in Gibson \cite{Gibson}. Thus if we can find a unimodular germ $g$ in $J^k(n-r,p-r)$ (for a large enough $k$), then the unimodular stratum in $J^k(n,p)$ will be given by a $r$-parameter unfolding of $g$. 

In fact in pair of dimensions lying in the boundary of the nice dimensions we may construct unfoldings of negative weights that are $C^0$-stable to get the unimodular stratum. Notice that the normal forms of these $C^0$-stable unfoldings are precisely those smooth $\fs K$-unimodular maps that are transverse to our stratification $\Sigma_{bnd}$; see for example Damon \cite{Damon4,Damon5} or  du Plessis and Wall \cite{duplessis3}. Using this idea we list in different cases the unimodular strata as unfoldings of unimodular germs in lower dimensions.

\subsection{Case 1: $(n\leq p)$}
\subsubsection{$\boldsymbol{(n,p) = (9,9)}$} Consider the following one parameter family of germs $f_{\lambda}: (\bb R^3,0) \to (\bb R^3,0)$ given by 
\begin{equation*}\label{unimod}
f_{\lambda}(x,y,z) = (x^2 + \lambda yz, y^2 + \lambda zx, z^2 + \lambda xy). \tag{$\star$}
\end{equation*}
By calculation of the $\fs K$-tangent space of $f_{\lambda}$ we find that except for $\lambda \neq 0, -8 ,1$, each $f_{\lambda}$ has $\fs K$-codimension $10$. Then, the $\fs K^k$-orbit of the $k$-jet of $f_{\lambda}$ in $J^k(3,3))$ for large enough $k$ also has codimension $10$ in $J^k(3,3)$ and the union of orbits of $f_{\lambda}$, parametrized by $\lambda$, will have codimension $9$. In fact, this stratum is a Zariski open set of the Thom-Boardman variety $\Sigma^{3,0}$. Thus it is clear that the $k$-jets $f_{\lambda}$ forms a unimodular stratum in $J^k(3,3)$. In fact, this is the only unimodular stratum of codimension $9$; see Wall \cite{Wall7,Wall6}.

Now, we consider the following unfolding $F_{\lambda}$ of $f_{\lambda}$. The unfolding $F_{\lambda} : \bb R^9 \to \bb R^9$ ($\lambda \neq 0,-8,1$),  which is a one-parameter family of maps, is given by:
$$F_\lambda(u_1,\ldots,u_6,x,y,z) = (u_1,\dots,u_6,u,v,w)$$
where
\begin{align*}
u & = x^2 + \lambda yz + u_1 y +u_2z\\
v & = y^2 + \lambda zx + u_3 x + u_4 z\\
w & = z^2 + \lambda xy + u_5x + u_6z
\end{align*}

\subsubsection{$\boldsymbol{(n,p) = (15,16)}$}

The unimodular stratum is related to that in the previous case in the following way. A result of Serre and Berger as presented in Proposition 2 of Eisenbud \cite{Eisenbud}  says that:

\begin{lem} \label{Serre-Berger} If $f : (\bb R^n,0) \to (\bb R^n,0)$ is a germ, then the algebra $\al Q(f) = \al E_n/\al I(f)$ (where $\al I(f)$ is the ideal generated by the components of $f$) has a unique minimal non-zero ideal. This ideal is given by the residue class of the ideal generated by the Jacobian $J(f)$ of $f$ in $\al E_n$.
\end{lem} 

Minimality implies that the residue class of $J^2(f)$ in $\al Q(f)$ is $0$. Consider the germ $f_{1\lambda} : (\bb R^3,0) \to (\bb R^4,0)$ given by $$f_{1\lambda}(x,y,z) = (f_{\lambda}(x,y,z),xyz)$$ where $f_{\lambda}$ is the map (\ref{unimod}) from the previous case. Now, observe that:
$$\fs K\text{-cod}(f_{1\lambda}) = \fs K\text{-cod}(f_{\lambda}) + \fs C\text{-cod}(f_{\lambda}) - 2 = 16$$ 
where $\fs C\text{-cod}(f_{\lambda}) = \dim \frac{\theta(f_{\lambda})}{\al I(f_{\lambda})\theta(f_{\lambda})}$. 

The unimodular stratum is a $12$-parameter unfolding of the $f_{1\lambda}$.
		
\subsubsection{$\boldsymbol{(n,p) = (21,23)}$}

In this case the unimodular germ is an unfolding of the germ $f_{2\lambda} : (\bb R^3,0) \to (\bb R^5,0)$ given by $f_{2\lambda}(x,y,z) = (f_{1\lambda}(x,y,z), 0)$.

\subsubsection{$\boldsymbol{(n,p) = (27,30)}$}	
In this case the unimodular germ is an unfolding of the germ $f_{3\lambda} : (\bb R^3,0) \to (\bb R^6,0)$ given by $f_{3\lambda}(x,y,z) = (f_{2\lambda}(x,y,z), 0)$. 	

\begin{rem}
Notice in general that:
$$\fs K\text{-cod}(f_{i\lambda}) = \fs K\text{-cod}(f_{\lambda}) + (p-n)(\text{dim}_{\bb R} \al Q(f_{\lambda}) - 2)$$
\end{rem}

\subsubsection{$\boldsymbol{(n,p) = (6t+2,7t+1)\,\,\,\,\text{for}\,\,\, t\geq 5}$}
The unimodular stratum is the unfoldings of $f_0 : (\bb R^4,0) \to \bb (\bb R^8,0)$ given by:
$$f(x,y,z,w) = (u_1,\ldots,u_8)$$
where
\begin{align*}
u_1 &= x^2 + y^2 + z^2\\
u_2 &= y^2 +\lambda z^2+ w^2\\
u_3 &= xy\\
u_4 &= xz\\
u_5 &= xw\\
u_6 &= yz\\
u_7 &= yw\\
u_8 &= zw
\end{align*}

\subsection{Case 2:  $n > p$}
\subsubsection{$\boldsymbol{(n,p) = (8,6)}$}
The smallest $(n,p)$ with $n > p$ in the boundary of the nice dimensions is $(8,6)$. The unimodular stratum is given by the following one-parameter family of maps $F_{\lambda} : \bb R^8 \to \bb R^6$:
$$F_{\lambda}(u_1,\ldots,u_4,x,y,z,w) = (u_1,\ldots,u_4,v,w)$$
where
\begin{align*}
v &= x^2 + y^2 + z^2 + u_1 y + u_2 w\\
w &= y^2 + \lambda z^2 + w^2 + u_3 x + u_4 z
\end{align*}

\subsubsection{$\boldsymbol{(n,p) = (10+k,7)\,\,\,\,\text{for}\,\,\, k\geq 0}$} 

For these dimensions we have to consider $f_{\lambda} : (\bb R^{9+k},0) \to (\bb R^7,0)$ given by
$F_{\lambda}(u_1,\ldots,u_6,w_1,\dots,w_k,x,y,z)=(u_1,\ldots,u_6,v)$, where
$$v = \sum_{i=0}^k w_i^2 + x^3+y^3+z^3+\lambda xyz + u_1 x + u_2 y + u_3z+u_4x^2 + u_5y^2 + u_6 z^2$$

\subsubsection{$\boldsymbol{(n,p) = (9,8)}$}

In this pair of dimensions the $\fs K$-unimodular stratum is given by the one parameter family $F : (\bb R^2 \times \bb R,0) \to (\bb R,0)$ defined by 
$F(x,y,\lambda) = x^4 + y^4 +\lambda x^2y^2$. It is easy to check that this is a $\fs K$-codimension $10$ family of germs and thus forms a stratum of codimension $9$.

\section{Bi-Lipschitz $\fs C$ and $\fs K$-triviality of deformations}

A bi-Lipschitz $\fs K$-equivalence of $r$-parameter deformations is a pair $(H,K)$ of bi-Lipschitz germs $H : (\bb R^r \times \bb R^n,0) \to (\bb R^r \times \bb R^n,0)$ and $K : (\bb R^r \times \bb R^n \times \bb R^p,0) \to (\bb R^r \times \bb R^n \times \bb R^p,0)$ with $H$ an $r$-parameter unfolding at $0$ of germ of the identity map of $\bb R^n$ such that the following diagram commutes.
\[
\xymatrixcolsep{3pc}\xymatrix{
	(\bb R^r \times \bb R^n,0)  \ar[r]^-{i} & (\bb R^r \times \bb R^n \times \bb R^p,0) \ar[r]^-{\pi} & (\bb R^r \times \bb R^n,0)\\
	(\bb R^r \times \bb R^n,0) \ar[u]_{H}\ar[r]^-{i} &(\bb R^r \times \bb R^n \times \bb R^p,0)\ar[u]^{K} \ar[r]^-{\pi}& (\bb R^r \times \bb R^n,0)\ar[u]^H\\
}
\]
Here $i$ is the canonical inclusion and $\pi$ is the canonical projection. Two $r$-parameter deformations $\Phi$ and $\Psi$ of $f$ are bi-Lipschitz $\fs K$-equivalent if there exist a bi-Lipschitz $\fs K$-equivalence $(H,K)$ as above such that:
$$K \circ (id,\Phi) = (id,\Psi) \circ H.$$

If $(H,K)$ has the special property that $H$ is the germ at $0$ of the identity mapping on $\bb R^m$, then $(H,K)$ is said to be a $\fs C$-equivalence and in that case $\Phi$ and $\Psi$ are said to be $\fs C$-equivalent deformations.

An $r$-parameter deformation $\Phi$ of a germ $f : (\bb R^m,0) \to (\bb R^p,0)$ is said to be bi-Lipschitz $\fs K$-trivial (resp. bi-Lipschitz $\fs C$-trivial) if it is bi-Lipschitz $\fs K$-equivalent (resp. bi-Lipschitz $\fs C$-equivalent) to the deformation $\Psi : (\bb R^r \times \bb R^m,0) \to (\bb R^p,0)$ given by $\Psi(u,x) = f(x)$. 

It is easy to prove that if $f , g : (\bb R^n,0) \to (\bb R^p,0)$  are two germs of smooth mappings such that there exist an invertible $(p\times p)$-matrix $(u_{ij})$ with entries germ of Lipschitz functions on $(\bb R^n,0)$ for which 
$$f_i = \sum_j u_{ij}g_j \,\,\,\,\, \text{for} \,\, \, 1 \leq i \leq p,$$ then $f$ and $g$ are bi-Lipschitz $\fs C$-equivalent. In fact, this can be proved by slighly modifying the argument in the smooth case (see 2.1 in Chapter V of Gibson \cite{Gibson}) and using the fact that a map is Lipschitz if and only if its components are Lipschitz.  

A sufficient condition for bi-Lipschitz $\fs C$-triviality is the following Thom-Levine type lemma which follows from the fact that the integration of a Lipschitz vector field gives a bi-Lipschitz flow: 
\begin{lem}\label{CThomLevine}
Let $F : (\bb R^n \times \bb R,0) \to (\bb R^p,0)$ be a one-parameter deformation of $f : (\bb R^n,0) \to (\bb R^p,0)$. If there exists a $(p\times p)$-matrix $(a_{ij})$ (not necessarily invertible) with entries germs of Lipschitz functions on $(\bb R^n \times \bb R,0)$ such that 
\begin{align}
\label{TLcondition}
\begin{bmatrix}
\frac{\partial F_1}{\partial \lambda}\\
\vdots \\
\frac{\partial F_p}{\partial \lambda}
\end{bmatrix}
&=
\begin{bmatrix}
a_{11} & \ldots & a_{1p} \\
\vdots & \ddots & \vdots \\
a_{p1} & \ldots & a_{pp} 
\end{bmatrix}
\begin{bmatrix}
F_1 \\
\vdots\\
F_p
\end{bmatrix},\tag{\dag}  
\end{align}
then $F$ is a bi-Lipschitz $\fs C$-trivial deformation.	
\end{lem}

In fact, if the condition \ref{TLcondition} holds, then the bi-Lipschitz $\fs C$-trivialization is given by integrating the following vector field:
$$X(x,y,\lambda) = \frac{\partial}{\partial \lambda} + \sum_{i=1}^p X_{i}(x,y,\lambda) \frac{\partial}{\partial y_i},$$
where $X_i(x,y,\lambda) = \sum_{j=1}^{p} a_{ij} y_j$ and $x_i$ and $y_j$ are coordinates on $\bb R^n$ and $\bb R^p$ respectively.

The converse of the above lemma, i.e. if bi-Lipschitz triviality of a deformation implies the existence a matrix of Lipschitz functions such that the condition (\ref{TLcondition}) holds, seems to be not known. It is therefore reasonable to say that a one-parameter deformation is \emph{strongly bi-Lipschitz $\fs C$-trivial} if there exists a $(p\times p)$-matrix (not necessarily invertible) of Lipschitz functions such that the condition (\ref{TLcondition}) is true.

The Thom-Levine criterion for the bi-Lipschitz $\fs K$-triviality is slighly more involved. First notice that the definition of the bi-Lipschitz triviality of a one-parameter deformation $F(x,t) = (F_1,\ldots,F_p) : (\bb R^n \times \bb R,0) \to (\bb R^p,0)$ of $f : (\bb R^n,0) \to (\bb R^p,0)$ amounts to saying that there exists a bi-Lipschitz one-parameter unfolding (of the germ of the identity map on $\bb R^n$) $H : (\bb R^n \times \bb R,0) \to (\bb R^n \times \bb R,0)$ and an invertible $(p \times p)$-matrix $b_{ij}$ of germs of Lipschitz functions on $\bb R^n \times \bb R$ such that:
$$\begin{bmatrix} 
f_1 \\
\vdots\\
f_p 
\end{bmatrix}
=\begin{bmatrix}
b_{11} & \ldots & b_{1p} \\
\vdots & \ddots & \vdots \\
b_{p1} & \ldots & b_{pp} 
\end{bmatrix}
\begin{bmatrix}
F_1 \circ H\\
\vdots\\
F_p \circ H
\end{bmatrix}
$$
Then, we claim that:
\begin{lem}\label{KTLlem}
Let $F : (\bb R^n \times \bb R,0) \to (\bb R^p,0)$ be a one-parameter deformation of $f : (\bb R^n,0) \to (\bb R^p,0)$. If there exists a $(p\times p)$-matrix $(a_{ij})$ (not necessarily invertible) with entries germs of Lipschitz functions on $(\bb R^n \times \bb R,0)$ and a germ of a Lipschitz vector field $X$ of the form
$$X = \frac{\partial}{\partial t} + \sum_{i=1}^n X_i(x,t)\frac{\partial}{\partial x_i}$$ with $X_i(x,0) = 0$
such that 
\begin{align}
\label{KTLcondition}
\begin{bmatrix}
X.F_1\\
\vdots \\
X.F_p
\end{bmatrix}
&=
\begin{bmatrix}
a_{11} & \ldots & a_{1p} \\
\vdots & \ddots & \vdots \\
a_{p1} & \ldots & a_{pp} 
\end{bmatrix}
\begin{bmatrix}
F_1 \\
\vdots\\
F_p
\end{bmatrix},\tag{\ddag}  
\end{align}
Then, $F$ is bi-Lipschitz $\fs K$-trivial deformation.
\end{lem}

The proof follows from the fact the integration of a Lipschitz vector fields gives a bi-Lipschitz flow. In fact, the bi-Lipschitz trivialization in the source is given by integrating the vector field $X$ and that in the product is given by the integration of
$$W(x,y,t) = \frac{\partial}{\partial t} + \sum_{i=1}^p W_{i}(x,y,t) \frac{\partial}{\partial y_i},$$ where $W_i(x,y,t) = \sum_{j=1}^{p} a_{ij} y_j$. The converse of the above lemma is not known and so we say that a one-parameter deformation is \emph{strongly bi-Lipschitz $\fs K$-trivial} if there exists a $(p \times p)$-matrix (not necessarlity invertible) of Lipschitz functions and a vector field $X$ such that the condition (\ref{KTLcondition}) holds. 

\section{Bi-Lipschitz $\fs K$-triviality of unimodular strata}

The aim of this section is to show that the unimodular strata given as one parameter deformations of unfoldings of certain germs in the Section \ref{secUnimodular} are bi-Lipschitz $\fs K$-trivial. We consider the two cases $n \leq p$ and $n > p$ separately.

To prove that the unimodular stratum in the boundary of the nice dimensions is bi-Lipschitz $\fs K$-trivial we show in different cases that the deformation that defines the unimodular stratum is strongly bi-Lipschitz $\fs C$-trivial. Actually if a unimodular stratum in the dimensions $(n,p)$ is given by the unfolding of a map $f: (\bb R^{n -r} \times \bb R,0) \to \bb (\bb R^{p-r},0)$ then the bi-Lipschitz triviality of the stratum follows from that of $f$. This can be seen as follows:

Let $F : (\bb R^n \times \bb R^r,0) \to (\bb R^p \times \bb R^r,0)$ be an unfolding of a germ $f : (\bb R^n,0) \to (\bb R^p,0)$. Then we know that $F$ is $\fs K$-equivalent to a suspension of $f$. Now, if $\tilde{f} : (\bb R^n \times \bb R,0) \to (\bb R^p,0)$ is a bi-Lipschitz $\fs K$-trivial deformation of $f$, then for any $\lambda,\lambda' \in \bb R$, $f_{\lambda}$ is bi-Lipschitz $\fs K$-equivalent to $f_{\lambda'}$. But then, $f_{\lambda} \times id$ will bi-Lipschitz $\fs K$-equivalent, to $f_{\lambda'} \times id$ (here $id$ is the germ of identity on $\bb R^r$). This implies $F_{\lambda}$ is bi-Lipschitz $\fs K$-equivalent to $F_{\lambda'}$, that is $F(\cdot,\lambda)$ is bi-Lipschitz $\fs K$-trivial.

In what follows we denote by $\al E_n$ the ring of germs of smooth functions on $(\bb R^{n-1} \times \bb R,0)$ and its maximal ideal by $\ak m_n$. We now show in different cases that the unimodular stratum is bi-Lipschitz $\fs C$-trivial. 

\subsection{The case $\boldsymbol{n \leq p}$}. In this case the finite $\fs K$-determinacy implies finite $\fs C$-determinacy since our maps are finite maps, i.e. the $\bb R$-vector space dimension of the algebra $\al Q(f) = \al E_n/{\al I(f)}$\footnote{Here $\al I(f)$ is the ideal generated by the components of $f$.} is finite; see Wall \cite{Wall5}. Thus it is enough to show that our deformations are bi-Lipschitz $\fs C$-trivial.

First is the equidimensional case $(9,9)$. In this case the unimodular stratum is the unfolding of $f_{\lambda}(x,y,z) = (x^2 + \lambda yz, y^2 + \lambda zx, z^2 + \lambda xy)$. We show that:

\begin{lem}\label{lem1} The one parameter family of germs $F: (\bb R^3 \times \bb R,0) \to (\bb R^3,0)$ given by $$F(x,y,z,\lambda) = (x^2 + \lambda yz, y^2 + \lambda zx, z^2 + \lambda xy)$$ is strongly bi-Lipschitz $\fs C$-trivial.
\end{lem}

\begin{proof} Let $\al I$ be the ideal of $\al E_4$ generated by the components $F_1,F_2,F_3$ of $F$, i.e.
$$\al I = \langle x^2 + \lambda yz, y^2 + \lambda xz + z^2 + \lambda xy \rangle.$$ Since $F_{\lambda}$ is a one-parameter deformation of a homogeneous polynomial of degree $2$. We can prove that $\ak m_3^4 \al E_4 \subset \al I$. In fact, $\al I.\ak m_3^2 \al E_4 = \ak m_3^4 \al E_4$.

Now, consider the following control function
$\rho (x,y,z,\lambda) = \sqrt{F_1^2 + F_2^2 + F_3^2}.$ Since $F_{\lambda} = F(x,y,z,\lambda)$ is finitely $\fs C$-determined for all but finitely many $\lambda$, there exist constants $c$ and $c'$ such that 
\begin{align}\label{1}
c'||(x,y,z)||^2 \leq \rho(x,y,z,t) \leq c. ||(x,y,z)||^2,
\end{align}
this is proved in Ruas \cite{Ruas}.

Since $\rho^2$ is of degree $4$ and $\frac{\partial F}{\partial \lambda}$ is of degree $2$ there exists a $(3\times 3)$-matrix $a_{ij}$ with entries in $\ak m_3^4 \al E_4$ such that,
\begin{align}
\rho^2(x,y,z,\lambda)
\begin{bmatrix}
\frac{\partial F_1}{\partial \lambda}\\
\frac{\partial F_2}{\partial \lambda} \\
\frac{\partial F_3}{\partial \lambda}
\end{bmatrix}
&=
\begin{bmatrix}
a_{11} & a_{12} & a_{13} \\
a_{21} & a_{22} & a_{23} \\
a_{31} & a_{32} & a_{33} 
\end{bmatrix}
\begin{bmatrix}
F_1 \\
F_2\\
F_3
\end{bmatrix}\label{eq2},  
\end{align}


Now consider the germ of the vector field $V$ on $(\bb R^3 \times \bb R^3 \times \bb R)$ at $0$ defined by:
$$V = \frac{\partial}{\partial \lambda} + \frac{1}{\rho^2} \left \{\sum_{j=1}^3 a_{1j}Y_j  \frac{\partial}{\partial Y_1} + \sum_{j=1}^3 a_{2j} Y_j \frac{\partial}{\partial Y_2} + \sum_{j=1}^3 a_{3j} Y_j \frac{\partial}{\partial Y_3} \right \}$$
where $(Y_1,Y_2,Y_3)=Y$ are coordinates on the target. 

We show that there exist a germ of a Lipschitz vector field $V'$, which coincides on a conical neighbourhood of the graph of $F$, whose integration will give the required bi-Lipschitz $\fs C$-trivialization of $F$. First notice that since $a_{ij}$ and $\rho^2$ are both of degree $4$ and $\rho^2$ is a control function also of degree $4$, $\frac{a_{ij} Y_j}{\rho^2}$ is continuous for any $i,j$. However, the derivative with respect to $x$ of $\frac{a_{ij}Y_j}{\rho^2}$ is $Y_j(\frac{\rho^2 \partial_x a_{ij} - 2 \rho a_{ij} \partial_x \rho }{\rho^4})$ that can't be bounded because the numerator of $\frac{\rho^2 \partial_x a_{ij} - 2 \rho a_{ij} \partial_x \rho }{\rho^4}$ is of degree degree $7$ while the denominator is of degree $8$. We modify $V$ as follows:

Let $D_1 = \{|Y| \leq C_1 |(x,y,z,t)|\}$ and $D_2 = \{|Y| \geq C_2|(x,y,z,t)|\}$ be cones in $(\bb R^3 \times \bb R) \times \bb R^3)$ with $C_1 < C_2$ and consider the function defined by:
$$p((x,y,z,t,Y))  = \begin{cases}
1 & \text{ if } (x,y,z,t,Y) \in D_1 \\
0 & \text{ if } (x,y,z,t,Y) \in D_2 \\
\end{cases} 
$$
and $0 < p(x,y,z,t,Y) < 1 $ if $C_1 |(x,y,z,t)| < |Y| < C_2 |(x,y,z,t)|$, such that the derivative of $p(x,y,z,t)Y$ with respect to any coordinate is bounded by a real number $K$. Now, consider the vector field $V' = p(x,y,z,t,Y) \times V$. Then, $V'$ coincides with $V$ on $D_1$. Notice that the derivative of components of $V'$ with respect to $Y_1,Y_2$ and $Y_3$ are bounded. By the construction of $p$, it is also clear that the derivative of components of $V'$ with respect to $(x,y,z)$ are also bounded on the neighbourhood $D_1$. Thus, $V'$ is a germ of a Lipschitz vector field.

Thus, integration of $V'$ will give a bi-Lipschitz $\fs C$-trivialization of $F$ by the Thom-Levine critierion \ref{CThomLevine}. This completes the proof.
\end{proof}

\begin{rem}
The modification of the vector field to get the required bi-Lipschitz vector field on a conical neighbourhood of the graph of the function $F$ above is inspired by a similar construction given in Ruas \cite{Ruas} for $C^k$-triviality.
\end{rem}

We next consider the case $(15,16)$. Recall from Section 3 that the unimodular stratum is the unfolding of a germ whose associated algebra is given by the ideal of generated by $\al I$ of Lemma \ref{lem1} and the determinant of the Jacobian ideal of \ref{unimod}. We show that:

\begin{lem}
The one parameter deformation $G : (\bb R^3 \times \bb R,0) \to (\bb R^4,0)$ given by 
$$G(x,y,z,\lambda) = (x^2 + \lambda yz, y^2 + \lambda xz, z^2 +\lambda xy, xyz)$$
is strongly bi-Lipschitz $\fs C$-trivial.
\end{lem} 

\begin{proof}
Write the components of $G$ as $G_1, G_2, G_3$ and $G_4$ and set $F = (G_1,G_2,G_3)$ the germ of smooth function from $(\bb R^3 \times \bb R,0) \to (\bb R^3,0)$. Consider the ideal $\al I$ generated by $G_1,G_2$ and $G_3$. Observe that $G_4$ is the determinant of the Jacobian $J(F)$ of $F$. 
By the Serre-Berger criterion \ref{Serre-Berger}, since the residue class of the square of the Jacobian of $G$ is $0$ in $\al Q(F)$, $J^2(F) = x^2y^2z^2 \in \al I$, i.e. there exist germs $c_1, c_2, c_3$ such that 
\begin{equation}\label{eq4}
c_1 G_1 + c_2 G_2 + c_3 G_3 + J^2(F) = 0. 
\end{equation}
Now, consider the $(4\times 4)$-matrix $M$ given by:
$$M = \begin{bmatrix} a_{11} & a_{12} & a_{13} & 0 \\
a_{22} & a_{22} & a_{23} & 0 \\
a_{31} & a_{32} & a_{33} & 0 \\
c_1 & c_2 & c_3 & J(F)
\end{bmatrix}$$ 
where $a_{ij}$ are entries of the $(3\times 3)$-matrix in Equation \ref{eq2} of Lemma \ref{lem1}. We put $\rho = \sqrt{G_1^2 + G_2^2 + G_3^2 }$ and calculate $\rho^2 \nabla G$. By Equation \ref{eq2} and \ref{eq4}, we have:
\begin{align}
\rho^2(x,y,z,\lambda)
\begin{bmatrix}
\frac{\partial G_1}{\partial \lambda}\\
\frac{\partial G_2}{\partial \lambda} \\
\frac{\partial G_3}{\partial \lambda} \\
\frac{\partial G_4}{\partial \lambda}
\end{bmatrix}
&=
M.
\begin{bmatrix}
G_1 \\
G_2\\
G_3 \\
J(F)
\end{bmatrix}\label{eq34},  
\end{align}
Then, modifying $\rho$ by multiplying with a map using a bump function as in previous Lemma we can show that there exists a Lipschitz vector field with the required properties. By the Thom-Levine criterion \ref{CThomLevine} it follows that $G$ is strongly bi-Lipschitz $\fs C$-trivial.
\end{proof}

The cases $(21,23)$ and $(27,30)$ follow similarly as the algebra of unimodular germs in this case coincides with the case $(15,16)$.

The last case is $\sigma(n,p)= 6(p-n)+8$ for $ p - n \geq 4 \text{ and } n \geq 4$ gives $(n,p) \in \{(6t+2,7t+1) : t \geq 5\}$. We show that 

\begin{lem}
The one-parameter deformation $F : (\bb R^4\times \bb R,0) \to \bb (\bb R^8,0)$ given by:
$$F(x,y,z,w,\lambda) = 
(x^2 + y^2 + z^2, y^2 +\lambda z^2+ w^2,xy,xz,xw,yz,yw,zw)$$
is strong bi-Lipschitz $\fs C$-trivial.
\end{lem}

\begin{proof}
Let $\al I$ be the ideal of $\al E_5$ generated by the components $F_1,\ldots,F_8$ of $F$, i.e. $$\al I = \langle x^2 + y^2 + z^2, y^2 +\lambda z^2+ w^2,xy,xz,xw,yz,yw,zw \rangle$$
Since $F_{\lambda}$ is a homogeneous polynomial of degree $2$ for each $\lambda$, we can prove that $\ak m^4_4\al E_5 \subset \al I$. In fact, $\al I.\ak m^2_4 \al E_5 = \ak m^4_4 \al E_5$. Then we can follow the proof of Lemma \ref{lem1} to show that there exists a control function $\rho: (\bb R^4\times \bb R) \to (\bb R,0)$ and an $(8 \times 8)$-matrix $(a_{ij})$ of Lipschitz function germs on $(\bb R^4\times \bb R,0)$ such that:
$$\begin{bmatrix} \frac{\partial F_1}{\partial \lambda} \\ \vdots \\ \frac{\partial F_8}{\partial \lambda}  \end{bmatrix} = \begin{bmatrix} a_{11} & \cdots & a_{18}\\
\vdots & \ddots & \vdots \\
a_{81} & \cdots & a_{88}
\end{bmatrix} \begin{bmatrix} F_1 \\  \vdots \\ F_8\end{bmatrix}$$
The Lemma then follows from the Thom-Levine criterion \ref{CThomLevine}.
\end{proof}

\subsection{The case $\boldsymbol{n > p}$}

The first pair of dimensions in this case is $(8,6)$ and the unimodular stratum is given by an unfolding of the one parameter deformation $F(x,y,z,w,\lambda) = (x^2+y^2+z^2, y^2 + \lambda z^2 + w^2).$
We show that:
\begin{lem}
The one parameter family of germs $F : (\bb R^4\times \bb R,0) \to (\bb R^2,0)$ defined by 
$$F(x,y,z,w,\lambda) = (x^2+y^2+z^2, y^2 + \lambda z^2 + w^2).$$ is strongly bi-Lipschitz $\fs K$-trivial.
\end{lem}

\begin{proof}Write $F = (F_1,F_2)$. The Jacobian of $F$ is,
$$DF = 
\begin{bmatrix}
2x & 2y & 2z & 0\\
0 & 2y & 2\lambda z & 2w	
\end{bmatrix}
.$$ The ideal of $\al E_5$ generated by $2\times 2$-minors, say $\Delta_1,\ldots,\Delta_6$, of $DF$ is 
$$JF = \langle \Delta_1,\ldots,\Delta_6\rangle=\langle xy, \lambda xz, xw, (\lambda-1)yz,yw,zw\rangle.$$
	
Since $F$ is a family of finitely $\fs K$-determined homogeneous polynomials of degree $2$, the $\fs K$-tangent space of $F$ contains a power of the maximal ideal $\ak m_5$. In fact, for $k \geq 3$ we have: 
\begin{equation}\label{Keq}
tF(\ak m_4^k \al E_5)  +   F^{*}(\ak m_2)(\ak m_4^{k-1}\al E_5) = \ak m_4^{k+1}. \al E_5
\end{equation}
where $F^*({\ak m_2})$ is the ideal generated by $F_1$ and $F_2$ in $\al E_5$ and $tF(\ak m^k_5\al E_5)$ is the ideal generated by $\ak m^kJF$. 
Now consider the control function $$\rho(x,y,z,w,\lambda) = \Delta^2_1 + \cdots + \Delta^2 + F_1^2 + F_2^2.$$
Since $\rho$ is of degree $4$ and $\frac{\partial F}{\partial \lambda}$ is of degree $2$, by Equation \ref{Keq} for $k = 5$, there exist a $(2\times 2)$-matrix $(A_{ij})$ with entries in $\ak m_4^5\al E_5$ and germs $\{X_i\}_{i=1}^4$ in $\ak m_5^4$ such that:
\begin{equation}\label{eq6}
\rho(x,y,z,w,\lambda)\begin{bmatrix}
\frac{\partial F_1}{\partial \lambda}\\
\frac{\partial F_2}{\partial \lambda}
\end{bmatrix}
=
X.\begin{bmatrix}
F_1 \\ F_2
\end{bmatrix}
+
\begin{bmatrix}
A_{11} & A_{12} \\
A_{21} & A_{22}
\end{bmatrix}
\begin{bmatrix}
F_1 \\
F_2
\end{bmatrix}
\end{equation}  
where $X = X_1 \frac{\partial}{\partial x} + X_2 \frac{\partial}{\partial y}+X_3 \frac{\partial}{\partial z}+X_4 \frac{\partial}{\partial w}$.
Dividing by $\rho$ on both sides of the Equation \ref{eq6} and rearranging the terms we get:
\begin{equation}
(\frac{\partial}{\partial \lambda}+\frac{X}{\rho}).
\begin{bmatrix}
F_1 \\ F_2
\end{bmatrix}
= 
\begin{bmatrix}
a_{11} & a_{12} \\
a_{21} & a_{22}
\end{bmatrix}
\begin{bmatrix}
F_1 \\ F_2
\end{bmatrix}
\end{equation}
where $a_{ij} = \frac{A_{ij}}{\rho}$. Now modifying $X$ and $(a_{ij})$, if necessary, as in previous lemmas, we can find a Lipschitz vector field $Y$ of the form $$\frac{\partial}{\partial \lambda} + Y_1 \frac{\partial}{\partial x} + Y_2 \frac{\partial}{\partial y}+Y_3 \frac{\partial}{\partial z}+Y_4 \frac{\partial}{\partial w}$$ and a $(2\times 2)$-matrix $(b_{ij})$ of Lipschitz functions such that:
$$Y.\begin{bmatrix}
F_1 \\ F_2
\end{bmatrix} = 
\begin{bmatrix}
	b_{11} & b_{12} \\
	b_{21} & b_{22}
\end{bmatrix}
\begin{bmatrix}
	F_1 \\ F_2
\end{bmatrix}
.$$
The Thom-Levine criterion \ref{KTLlem} then shows that $F$ is a bi-Lipschitz $\fs K$-trivial deformation. This completes the proof.
\end{proof}

The next case is the pair $(9,8)$ where the unimodular statum is an unfolding of the one parameter deformation given by $F(x,y,\lambda) = x^4 + y^4 + \lambda x^2y^2$. We have:

\begin{lem} The one parameter family of germs $F : (\bb R^2\times \bb R,0) \to (\bb R,0)$ defined by 
$$F(x,y,\lambda) = x^4 + y^4 + \lambda x^2y^2$$ is strongly bi-Lipschitz $\fs K$-trivial.
\end{lem}

\begin{proof}
Follows from Fernandes and Ruas \cite{Ruas1}.
\end{proof}

Finally, in the last case we have:

\begin{lem} The one parameter family of germs $F : (\bb R^{3+k}\times \bb R,0) \to (\bb R,0)$ given by
$$F(w_1,\ldots,w_k,x,y,z,\lambda) = \sum_{i = 0}^k w_i^2 + x^3 + y^3 + z^3 + \lambda xyz$$ is strongly bi-Lipschitz $\fs K$-trivial.
\end{lem}

\begin{proof}
Follows from Birbrair et al. \cite{Birbrair2}.
\end{proof}

We can summarize the results obtained in this section into the following Theorem:

\begin{thm}
The Thom-Mather stratification in the boundary of the nice dimensions is bi-Lipschitz $\fs K$-invariant stratification.
\end{thm}

\begin{rem}
The bi-Lipschitz $\fs K$-triviality of the unimodular strata in the boundary of the nice dimensions can also be proved by applying the main result of Ruas and Valette \cite{RuasValette}. Also, For $n \geq p$ one can use the result of Saia et al. \cite{Saiaetal} for bi-Lipschitz $\fs C$-triviality. Our method provides stronger results for
we explicitly construct the trivializations using the Thom-Levine criterion. 
\end{rem}


\subsection*{Acknowledgements} The first author was supported by FAPESP grant 2014/00304-2 and CNPq grant 306306/2015-8. The second author was supported by FAPESP grant 2015/12667-5.

\end{document}